\documentclass[12pt]{amsart}
\usepackage{caption}
\usepackage{etex}
\usepackage{amsmath, amssymb}
\usepackage[frame,cmtip,arrow,matrix,line,graph,curve]{xy}
\usepackage{graphpap, color, paralist, pstricks}
\usepackage[mathscr]{eucal}
\usepackage[pdftex]{graphicx}
\usepackage[pdftex,colorlinks,backref=page,citecolor=blue]{hyperref}
\usepackage{tikz}
\usepackage{kotex}
\usepackage{array}
\newcolumntype{P}[1]{>{\centering\arraybackslash}p{#1}}
\newcolumntype{M}[1]{>{\centering\arraybackslash}m{#1}}
\usepackage{booktabs}
\usepackage{float} 
\usepackage{multirow}

\newtheorem{theorem}{Theorem}[section]
\newtheorem{proposition}[theorem]{Proposition}
\newtheorem{corollary}[theorem]{Corollary}
\newtheorem{lemma}[theorem]{Lemma}

\theoremstyle{definition}

\newtheorem{example}[theorem]{Example}

\newtheorem{remark}[theorem]{Remark}

\renewcommand{\AA}{\mathbb{A}}
\newcommand{\PP}{\mathbb{P}}
\newcommand{\QQ}{\mathbb{Q}}
\newcommand{\CC}{\mathbb{C}}

\newcommand{\ZZ}{\mathbb{Z}}

\newcommand{\cO}{\mathcal{O} }

\newcommand{\cC}{\mathcal{C} }
\newcommand{\cE}{\mathcal{E} }
\newcommand{\cF}{\mathcal{F} }

\newcommand{\cM}{\mathcal{M} }

\newcommand{\cL}{\mathcal{L} }

\newcommand{\cS}{\mathcal{S} }

\newcommand{\cU}{\mathcal{U} }

\newcommand{\rH}{\mathrm{H} }

\newcommand\wt{\mathrm{wt}}

\newcommand\bB{\mathbf{B}}
\newcommand\bE{\mathbf{E}}

\newcommand\bH{\mathbf{H}}

\newcommand\bM{\mathbf{M}}

\newcommand\bR{\mathbf{R}}

\newcommand\bZ{\mathbf{Z}}

\newcommand\Pic{\mathrm{Pic}}

\newcommand{\ch}{\text{ch}}
\newcommand{\td}{\text{td}}

\newcommand\Hom{\mathrm{Hom} }
\newcommand\Ext{\mathrm{Ext} }
\newcommand\cExt{\mathcal{E}xt }
\newcommand{\Gr}{\mathrm{Gr} }

\newcommand\GL{\mathrm{GL}}

\newcommand\lr{\rightarrow}

\newcommand{\ses}[3]{0\lr{#1}\lr{#2}\lr{#3}\lr 0}

\newcommand\rk{\mathrm{rk}}

\def\Sp{\mathrm{Sp}}

\def\GL{\mathrm{GL}}

\def\vir{\mathrm{vir}}

\def\fix{\mathrm{fix}}
\def\mov{\mathrm{mov}}

\hypersetup{%
pdftitle={DT-GW correspondence on local Fano 3-folds},%
pdfauthor={Kiryong Chung},%
pdfkeywords={moduli of sheaves, elementary modification, moduli space},%
citecolor=blue,%
linkcolor=blue,%
}

\begin{document}

\title[GW-DT-GV correspondences on a local CY 4-fold]{Two-Dimensional Torus Action on the Quadric Threefold and GW–DT–GV Correspondence}

\author{Kiryong Chung}
\address{Department of Mathematics Education, Kyungpook National University, 80 Daehakro, Bukgu, Daegu 41566, Korea}
\email{krchung@knu.ac.kr}

\author{Sanghyeon Lee}
\address{Department of Mathematics, Ajou University, 206 World cup-ro, Yeongtong-gu, Suwon-si, Gyeonggi-do, 16499, Korea}
\email{sanghyeon25@ajou.ac.kr}


\begin{abstract}
We compute Gromov--Witten (GW) and Donaldson-Thomas (DT) invariants (including descendant invariants) for local CY $4$-folds over a smooth quadric threefold $Q$, up to degree $4$. We use the virtual localization formula with respect to a two-dimensional torus action to compute DT invariants. Combining these results with the computation of the J-function of the variety $Q$, we confirm the correspondence between DT and Gopakumar--Vafa invariants suggested by Cao--Maulik--Toda in genus zero. Based on these results, we predict a genus one BPS invariant on $Q$. 
\end{abstract}
\maketitle
\section{Introduction}\label{sec:intro}
\subsection{Introduction}
The geometry of compactified moduli spaces of rational curves
on linear sections of Grassmannians and various Fano threefolds
has been studied in \cite{CK11,CHK12,CHL18,CY19,Chu22}. These descriptions do more than identify the underlying moduli spaces: they provide universal families, resolve the behavior of boundary sheaves, and make intersection-theoretic calculations accessible. A natural enumerative application is the calculation of Donaldson--Thomas invariants on local Calabi--Yau fourfolds. These computations were carried out for the quintic del Pezzo threefold in \cite{CLW24} and for the prime Fano threefold of index one in \cite{CW26}. The present paper develops this approach for quartic curves on a smooth quadric threefold, where the stable sheaf compactification is no longer exhausted by structure sheaves of rational Cohen--Macaulay curves. We write $\mathbf M_d(Q)$ for the Simpson moduli space of stable pure one-dimensional sheaves with Hilbert polynomial $dm+1$. For $d\leq 3$, the low-degree geometry of $Q$ allows these spaces to be described in terms of structure sheaves of rational Cohen--Macaulay curves; see \cite{CHK12,CHY23}. The degree $d=4$ case is the first one on $Q$ for which this description
is insufficient. In addition to the rational component $\mathbf R_4$,
the moduli space $\mathbf M_4(Q)$ contains an elliptic component
$\mathbf E_4$, whose general point represents a line bundle
$\mathcal O_E(p)$ on a smooth elliptic quartic curve $E\subset Q$,
with $p\in E$. Thus one must distinguish a rational curve, a limit of rational curves in a Hilbert scheme, and the stable sheaf representing the corresponding boundary point in Simpson's compactification. The structure sheaf of a non-Cohen--Macaulay Hilbert limit with an embedded point is not pure. We must therefore modify the family to obtain a flat family of stable sheaves. This distinction is already familiar from the study of stable sheaves on quartic space curves \cite{CCM16}. Here it becomes essential to an explicit local Calabi--Yau fourfold calculation.

\subsection{Main results}\label{subsec:intro-results}
Let $h_1=c_1(\mathcal O_Q(1))$ and let $h_2=[\ell]$ be the Poincar\'e dual of a line $\ell\subset Q$. Let $\mathcal F_{\mathrm{norm}}$ be the normalized universal
sheaf on $Q\times\mathbf M_d(Q)$, and let $\pi_Q$ and
$\pi_{\mathbf M}$ be the projection maps.
For $i=0,1$, define
\[
\tau_i(h_{2-i})
=
\pi_{\bM *}\left(
\pi_Q^*h_{2-i}\cdot
\left\{
\mathrm{ch}(\cF_{\mathrm{norm}})
\cdot \mathrm{td}(K_Q)^{-1}
\right\}_{i+2}
\right),
\]
where $\{\cdot\}_{i+2}$ denotes the component of
cohomological degree $2i+4$.
The corresponding primary and descendent invariants are
\[
 \left\langle\tau_i(h_{2-i})\right\rangle_d
 =
 \int_{[\mathbf M_d(Q)]^{\mathrm{vir}}}
 \tau_i(h_{2-i}),
 \quad i=0,1,
\]
where we use the orientation
$[\mathbf M_d(X)]^{\mathrm{vir}}
=(-1)^{3d-1}[\mathbf M_d(Q)]^{\mathrm{vir}}$ (\cite{CMT18,CT21}) for $X=\operatorname{Tot}(K_Q)$.
We also compute the genus-zero twisted Gromov--Witten invariants
$N_{0,d}$.
\begin{theorem}[\protect{TABLE \ref{tab:traq1}, Proposition \ref{pdin}}]\label{intro:main-theorem}
For $X=\operatorname{Tot}(K_Q)$, the invariants
in degrees $1\leq d\leq4$ are
\[
\setlength{\arraycolsep}{12pt}
\renewcommand{\arraystretch}{1.7}
\begin{array}{|c||c|c|c|}
 \hline
 d
 & \langle\tau_0(h_2)\rangle_d
 & \langle\tau_1(h_1)\rangle_d
 & N_{0,d}\\[4pt]
 \hline
 1 & 10 & 10 & 10\\[4pt]
 \hline
 2 & -175 & 0 & -172\frac{1}{2}\\[6pt]
 \hline
 3 & 5870 & -1310 & 5871\frac{1}{9}\\[6pt]
 \hline
 4 & -265940 & -73360 & -265983\frac{1}{8}\\[6pt]
 \hline
\end{array}
\]
In particular, the primary DT/BPS correspondence
of Cao--Maulik--Toda \cite{CMT18} holds up to degree $d=4$.
\end{theorem}
The DT-calculation uses the geometry of stable-sheaf
moduli and torus localization, while the genus-zero
Gromov--Witten calculation uses the $I$-function
and the mirror map.
Unlike the preceding calculations in \cite{CLW24,CW26},
the quartic calculation on $Q$ involves both an elliptic
component of the stable-sheaf moduli space and
positive-dimensional torus-fixed components.

\begin{corollary}\label{intro:cao-toda-corollary}
Assume that
\[
 n_{1,1}=n_{1,2}=n_{1,3}=0,
 \qquad n_{1,4}=45150.
\]
Then the Cao--Toda descendent formula \cite{CT21}
for $\langle\tau_1(H)\rangle_d$ holds for $1\leq d\leq4$.
\end{corollary}
The \texttt{Macaulay2} \cite{M2} code used for the localization computations in this paper is available from the authors upon request.

\subsection*{Acknowledgements}
The authors gratefully acknowledge the many helpful suggestions and comments provided by Jinwon Choi and Sukmoon Huh.

\section{Preliminary}

\subsection{Localization formula}
Let $T$ be an algebraic torus and $M$ be a $T$-space. We work in equivariant cohomology $H_T^*(-)$ and homology $\rH_*^T(-)$. Let us denote the equivariant Euler class by $e_T(-)$. 
\begin{theorem}\label{loccomp}
Let $Z$ be a connected component of $M^T$. 
Assume that the fixed part of the virtual tangent complex restricts to
the ordinary tangent bundle of $Z$ (that is, $(T_{M}^{\vir}|_Z)^{\fix}=T_Z$).
Let
\[
N_Z^{\vir}=(T_{M}^{\vir}|_Z)^{\mov}
\]
be the virtual normal bundle. Then the contribution of $Z$ to the
virtual class is
\[
i_{Z*}\left(
\frac{[Z]}{e_T(N_Z^{\vir})}
\right).
\]
Consequently,
\[
[M]^{\vir}
=
\sum_{Z\subset M^T}
i_{Z*}\left(
\frac{[Z]}{e_T(N_Z^{\vir})}
\right),
\]
where the sum runs over the connected components of $M^T$.
\end{theorem}

\subsection{Torus action on the quadratic $3$-fold $Q$}\label{subdefq}
Let us assume the $(t_1, t_2)\in T=(\CC^*)^2$-action on the vector space $V_4$ by \[(t_1,t_2)\cdot (x_0,x_1, x_2,x_3)=(t_1x_0,t_2x_1,t_2^{-1}x_2,t_1^{-1}x_3).\] Then the induced action on the space $\wedge^2 V_4$ fixes the quadratic $3$-fold:
\[
Q=\Gr(2, V_4)\cap H
\]
where the hyperplane $H\subset \PP(\wedge^2V_4)$ is defined by $p_{03}-p_{12}=0$.
\begin{lemma}\label{tacq}
Under a suitable coordinate change, the $T$-action $\rho$ in $H\cap \PP(\wedge^2V_4)=\PP_{\{z_i\}}^4$ is given by
\[
(t_1,t_2)\cdot [z_0:z_1:z_2:z_3:z_4]=[t_1t_2z_0:t_1t_2^{-1}z_1:z_2:t_1^{-1}t_2z_3:t_1^{-1}t_2^{-1}z_4],\;(t_1,t_2)\in T
\]
where the quadratic hypersurface $Q$ is defined by 
$$z_0z_4 + z_1z_3+z_2^2=0.$$ 
Furthermore, let $r=(0341)$ and $f=(13)$ be the permutations on the coordinates $z_i$. Then we have
\[
r\rho(t_1,t_2)r^{-1}=\rho(t_2^{-1}, t_1),\; f\rho(t_1,t_2)f^{-1}=\rho(t_2, t_1).
\]
\end{lemma}
\begin{proof}
Put $[z_0:z_1:z_2:z_3:z_4]=[p_{01}:p_{02}: p_{12}: -p_{13}: p_{23}]$. Then the quadratic equation comes from the Plucker relation of $\Gr(2, V_4)\subset \PP^5$. By a direct computation, one can easily check that the transformations $r$ and $f$ normalize the torus $T$.
\end{proof}
Throughout the localization computations, we will use the \emph{dihedral group} \(D_4=\langle r,f\rangle\) in Lemma \ref{tacq}. The $T$-fixed sheaves will be grouped into $D_4$-orbits, and it suffices to compute the localization contribution for one representative of each orbit.

\subsection{Irreducible rational curves in $Q$ stable under the torus action}


Note that $Q=Q_3$ is isomorphic to a generalized partial flag variety $\mathrm{LG}(2,4) \cong \Sp_4(\CC)/P$. 
In fact, we can show that all fixed irreducible reduced curves in $Q$ have degree $\le 2$, and they are described in Corollary \ref{coro:fixedline} below.
\[
\begin{tikzpicture}[baseline=-0.5ex, scale=4]

\fill (0,0) circle (0.5pt) node[below] {$p_2$};
\fill (0,1) circle (0.5pt) node[above] {$p_1$};
\fill (1,1) circle (0.5pt) node[above] {$p_0$};
\fill (1,0) circle (0.5pt) node[below] {$p_3$};
\draw (0,1)--(1,1);
\draw (0,0)--(1,0);
\draw (0,1)--(0,0);
\draw (1,1)--(1,0);
\draw[line width=1.5pt,gray] (0,0)--(1,1);
\draw[line width=1.5pt,gray] (0,1)--(0.48,0.52);
\draw[line width=1.5pt,gray] (0.52,0.48)--(1,0);
\end{tikzpicture}
\]
By the general theory on GKM (Goresky, Kottwitz, and MacPherson) manifolds in \cite{GHZ06}, \cite[Section 4.2.2]{GSZ12}, \cite[Section 3, Lemma 4.2]{FW01}, for any semisimple Lie group $G$ and a parabolic subgroup $P \subset G$ containing a maximal torus $T$, the natural action of $T$ on the generalized partial flag variety $G/P$ becomes a GKM-action. Moreover, the $T$-fixed points are finite, and there is a bijection $(G/P)^T \cong W_G/W_P$, where $W_G$ is the Weyl group of $G$, and $W_P$ is the Weyl group of the Levi subgroup of $P$.

Moreover, for a point $p_w$ corresponding to a point $[w] \in W_G/W_P$, there is a bijection between the set of 1-dimensional $T$-orbits passing through $p_w$ (or GKM edges passing through $p_w$) and $\Phi_G^+ \setminus \Phi_{L_P}^+$ where $\Phi_G$, $\Phi_G^+$ are the root system of $G$ and the set of positive roots of $G$ respectively, and $L_P$ is the Levi subgroup of the parabolic subgroup $P$.

\begin{proposition} \label{prop:Q3GKM}
For $Q \cong LG(2,4)=Sp_4/P$, there are $4$ $T$-fixed points and $6$ 1-dimensional $T$-orbits (=GKM edges). In particular, for any $T$-fixed point, there are 3 GKM edges passing through it.
\end{proposition}
\begin{proof}
It is well known that the Weyl group $W=W_{\Sp_4}$ is isomorphic to the dihedral group of order 8, and $W_P$ is isomorphic to $S_2$. Therefore, $|W/W_P|=4$, which means there are 4 $T$-fixed points.
\newcommand{\set}[1]{\left\{ #1 \right\}}
Let $\epsilon_1((t_1,t_2)) = t_1$, $\epsilon_2((t_1,t_2)) = t_2$ be two characters of $T=(\CC^*)^2$. Let us assume that $\epsilon_1 > \epsilon_2 > 0$. Then, we observe that the root system is $\Phi = \set{\pm 2\epsilon_1, \pm 2\epsilon_2, \pm(\epsilon_1 - \epsilon_2), \pm(\epsilon_1 + \epsilon_2)}$ and the set of positive roots is $\Phi^+ = \set{2\epsilon_1, 2\epsilon_2, \epsilon_1 + \epsilon_2, \epsilon_1 - \epsilon_2}$. 

Note that the Levi subgroup $L_P$ of $P$ is isomorphic to $\GL_2$. 
Let $\Phi_{L_P}$, $\Phi_{L_P}^+$ be the root system and the set of positive roots of the Levi subgroup $L_P$, respectively. Since the root system of $\GL_2$ is $\set{\pm(\epsilon_1-\epsilon_2)}$, $\Phi_{L_P}^+ = \set{\epsilon_1 - \epsilon_2}$. Thus $\Phi^+ \setminus \Phi_{L_P}^+ = \set{2\epsilon_1, 2\epsilon_2, \epsilon_1+\epsilon_2}$, $|\Phi^+ \setminus \Phi_{L_P}^+|=3$. Therefore, there are $3$-adjacent $1$-dimensional $T$-orbits for each $T$-fixed point, so that there are $(4\times3)/2 = 6$ 1-dimensional $T$-orbits.
\end{proof}

By direct computation, we can check the lines and conics $l_0,l_1,l_2,l_3,Q_0,Q_1$ appearing in Corollary \ref{coro:fixedline} are $T$-stable, reduced curves in $Q$. Then, the following corollary is directly obtained from Proposition \ref{prop:Q3GKM}.

\begin{corollary} \label{coro:fixedline}
The $T$-fixed irreducible rational curves in $Q\subset \PP(\wedge^2V_4)$ consist of the following:
\begin{enumerate}
\item The fixed lines in $Q$ are
\[
l_0:e_0\wedge (\alpha e_1+\beta e_2),\;l_1: e_1\wedge (\alpha e_0+\beta e_3),\; l_2:e_2\wedge (\alpha e_0+\beta e_3),\; l_3:e_3\wedge (\alpha e_1+\beta e_2)
\]
for $[\alpha:\beta]\in \PP^1$.
\item The smooth fixed conics in $Q$ are given by
\[
Q_0:(\alpha e_0+\beta e_1)\wedge (\alpha e_2+\beta e_3),\; Q_1:(\alpha e_0-\beta e_2)\wedge (\alpha e_1+ \beta e_3)
\]
for $[\alpha:\beta]\in \PP^1$.
\end{enumerate}
\end{corollary}

\subsection{Geometric materials of a smooth quadric threefold $Q$}
Recall that $V_4$ is the vector space of dimension $\dim V_4=4$.
The Fano threefold $Q$ is defined by the linear intersection $Q=\Gr(2,V_4)\cap H$ where $H$ is the hyperplane in $\PP(\wedge^2 V_4)=\PP^5$ (Section \ref{subdefq}).
\begin{itemize}
\item The integral cohomology ring of $Q$ is
\[
\rH^*(Q,\mathbb Z)
\cong
\frac{\mathbb Z[h_1,h_2]}
{\langle h_1^2-2h_2, h_2^2\rangle},
\;
\deg(h_i)=2i,
\]
where $h_1=c_1(\mathcal O_Q(1))$ and $h_2=[\ell]$ is the Poincar\'e dual of a line $\ell\subset Q$. In particular, $\int_Q h_1^3=2$.
\item We have $\Pic(Q)=\ZZ\langle h_1\rangle$ and $K_Q=-3h_1$.
\item The restriction $\cE:=\cU|_Q$ of the universal subbundle $\cU$ on $\Gr(2,V_4)$ is called the \emph{spinor bundle} on $Q$. The Chern classes of $\cE$ are given by
\[
\bigl(\rk(\cE),c_1(\cE),c_2(\cE)\bigr)
=
(2,-h_1,h_2),
\]
and $\cE$ fits into the exact sequence
\begin{equation}\label{stre}
\ses{\cE}{V_4\otimes\cO_Q}{\cE^\vee=\cE(1)}.
\end{equation}
\end{itemize}

\section{GW-invariants on $Q$}

In this section, we will compute genus $0$ GW invariants of the local Fano 4-fold $X:=|K_{Q}|$, which is the total space of the canonical line bundle over $Q$ up to degree $4$. Moreover, via the definition of \cite[Section 0.3, equation (2)]{KP08}, we will compute genus $0$ Gopakumar-Vafa invariants of $X$ up to degree $4$.

We will start with introducing the I-function of $Q_3$. 
By \cite[Section 0, equation (*)]{Giv98}, the I-function of the Fano 3-fold $Q_3$ is given by:
\begin{equation}\label{eq:Iftn1}
I_{Q_3} (q,z) = ze^{H\log q/z}
\sum_{d\geq 0}q^d
\frac{
\displaystyle\prod_{m=1}^{2d}(2H+mz)}
{\displaystyle\prod_{m=1}^{d}(H+mz)^5}.
\end{equation}
By \cite[Corollary 8(i)]{CG07}, the I-function of $Q_3$ coincides with the $J$-function of $Q_3$, $J_{Q_3}(q,z)$.

Then, by \cite[Equation (13)]{CG07}, the I-function of the local CY 4-fold $X$ is given by:
\begin{equation}\label{eq:Iftn2}
I_X(q,z) = ze^{H\log q/z}
\sum_{d\geq 0}q^d
\frac{
\displaystyle\prod_{m=1}^{2d}(2H+mz)
\displaystyle\prod_{m=0}^{3d-1}(-3H-mz)}
{\displaystyle\prod_{m=1}^{d}(H+mz)^5}.
\end{equation}
Note that the term in the numerator
$
\prod_{m=0}^{3d-1}(-3H-mz)
$
comes from the quantum Lefschetz property \cite[Section 7]{CG07}.

Let $x := H/z$. Then we can write 
\[
I(q,z) = z e^{H\log q/z}
\left(
1+f_1(q)x+f_2(q)x^2+f_3(q)x^3 + O(x^4)
\right).
\]
Then, by \eqref{eq:Iftn2} and direct computation, we obtain
\begin{align}\label{eq:f1}
f_1(q) = -12q+270q^2-11200q^3+606375q^4+O(q^5),
\end{align}
and
\begin{align}\label{eq:f2}
f_2(q) = -30q+\frac{1899}{2}q^2
-\frac{130600}{3}q^3
+\frac{19785675}{8}q^4
+O(q^5).
\end{align}

Now, by \cite[Theorem 4.8]{CCIT09}, when we define the mirror map $Q(q)$ by
\begin{equation}\label{eq:mirror}
\log Q = \log q + f_1(q), \ \ Q=q \, e^{f_1(q)},
\end{equation}
we have the following relation between the I-function and the J-function of the local CY 4-fold $X$.
\[
I_X(q(Q),z) = J_X(Q,z), \quad q = Q^{-1}.
\]	
Using \eqref{eq:f1} and \eqref{eq:mirror}, we have
\[
Q = q-12q^2+342q^3-14728q^4+O(q^5).
\]
And its inverse is
\[
q=Q+12Q^2-54Q^3+2848Q^4+O(Q^5).
\]
Note that we have
\[
e^{-f_1(q)x}
\left(1+f_1(q)x+f_2(q)x^2+\cdots\right)=
1+\left(f_2-\frac12f_1^2\right)x^2+\cdots.
\]
Now, by a direct computation, we have
\begin{equation}\label{eq:Jftn1}
J_X(Q,z) = z e^{H\log Q/z}
\left(
1+G_2(Q)\frac{H^2}{z^2}
+O(z^{-3})
\right)
\end{equation}
where $\ G_2(Q) = \left.
\left(f_2(q)-\frac12f_1(q)^2\right)
\right|_{q=q(Q)}$. A direct computation gives
\begin{equation}\label{eq:G2}
G_2(Q) = -30Q+\frac{1035}{2}Q^2
-\frac{52840}{3}Q^3
+\frac{6383595}{8}Q^4
+O(Q^5).
\end{equation}

On the other hand, by the quantum Lefschetz property \cite{KKP03}, the GW invariant of $X$, with the single cohomology insertion $h_2$ is equal to
\begin{equation*}
N_{0,d}(h_2) := \int_{[\overline M_{0,1}(Q_3,d)]^{\mathrm{vir}}} e(R^1\pi_* f^*K_{Q_3}) \cap ev^*h_2
\end{equation*}
for each degree $d$. Note that we used the fact $\overline M_{0,1}(Q_3,d) \cong \overline M_{0,1}(X,d)$.

Let us recall the definition of the $J$-function of $X = |K_{Q_3}|$. Let $i : Q_3 \to X$ be the zero section, and let $M_d=\overline M_{0,1}(Q_3,d)$, $\mathbb E_d=R^1\pi_*f^*K_{Q_3}$. Then the degree $d$ term of the $J$ function of $X$ is
\begin{equation}\label{eq:Jftndegreed}
(\operatorname{ev}_X)_*
\left(
\frac{e(\mathbb E_d)}
{z-\psi}
[M_d]^{\mathrm{vir}}
\right)=
i_*
(\operatorname{ev}_Q)_*
\left(
\frac{e(\mathbb E_d)}
{z-\psi}
[M_d]^{\mathrm{vir}}
\right)
\end{equation}
where $\text{ev}_Q$, $\text{ev}_X$ are the evaluation maps at the marked point, to $Q_3$ and $X$ respectively.

To express this in terms of cohomology classes on $Q_3$, we take the pull-back $i^*$. Since $e(N_{Q_3/X}) = -3H$, we have $i^*i_*(\alpha) = e(N_{Q/X})\alpha = -3H \alpha$.

Let 
$A_d
:=
(\operatorname{ev}_Q)_*
\left(
e(\mathbb E_d)\cap[M_d]^{\mathrm{vir}}
\right)
\in H^2(Q)
$. Since $H^2(Q)$ is generated by $H$, we can write $A_d = c_d H$. Recall that $h_2 = H^2/2$. We have
\[
N_{0,d}(h_2) =
\int_Q h_2\cup A_d =
c_d\int_Q\frac{H^3}{2} = c_d.
\]
Thus $A_d=N_{0,d}(h_2)H$.

Note that $\frac1{z-\psi}=\frac1z+O(z^{-2})$. Hence, by taking $i^*$ at \eqref{eq:Jftndegreed}, we observe that the $z^{-1}$ coefficient of the degree $d$ term of the twisted $J$-function is $i^*i_* A_d = e(K_Q) A_d = -3N_{0,d}(h_2)H^2$.

Therefore, comparing this with \eqref{eq:Jftn1}, we have 
\[
G_2(Q) = -3\sum_{d\geq1}N_{0,d}(h_2)Q^d.
\]
By \eqref{eq:G2}, we have
\[
\sum_{d\geq1}N_{0,d}(h_2)Q^d = 10Q-\frac{345}{2}Q^2
+\frac{52840}{9}Q^3
-\frac{2127865}{8}Q^4
+O(Q^5).
\]
Therefore the twisted genus zero GW invariants are
\begin{equation}\label{eq:GWtable1}
(N_{0,1},N_{0,2},N_{0,3},N_{0,4}) = (10, -\frac{345}{2}, \frac{52840}{9}, -\frac{2127865}{8}).
\end{equation}
On the other hand, by \cite[Equation (2)]{KP08}, we have
\[
\sum_{d\geq1}N_{0,d}(h_2)Q^d = \sum_{d\geq1}n_{0,d}(h_2)
\sum_{k\geq1}\frac{Q^{kd}}{k^2}.
\]
Recall that $n_{g, d}$ are genus $g$, degree $d$ Gopakumar-Vafa invariants. From \eqref{eq:GWtable1}, we have
\[
(n_{0,1},n_{0,2},n_{0,3},n_{0,4})=
(10,-175,5870,-265940).
\]

In summary, we have the following table for genus zero GW invariants and Gopakumar-Vafa (GV) invariants of $X=|K_{Q_3}|$.

\begin{table}[h]
\centering
\begin{tabular}{|c|c|c|}
\hline
$d$ & $N_{0,d}$ & $n_{0,d}$ \\
\hline
\hline
$1$ & $10$ & $10$ \\
\hline
\hline
$2$ & $-172 \frac{1}{2}$ & $-175$  \\
\hline
\hline
$3$ & $5871 \frac{1}{9}$ & $5870$  \\
\hline
\hline
$4$ & $ -265983 \frac{1}{8}$ & $-265940$ \\
\hline
\end{tabular}
\smallskip
\caption{Twisted GW and DT-type invariants of the quadratic threefold $Q$}
\label{tab:traq1}
\end{table}
\section{Global description of the space $\bM_4$}
Let $\bM=\bM_4(Q)$ be the moduli space of stable sheaves $F$ on $Q$ with Hilbert polynomial $\chi(F(m))=4m+1$.
\begin{proposition}\label{clm4}
Let $F$ be a stable sheaf on $Q$ such that $[F]\in \bM_4$. Then one of the following holds.
\begin{enumerate}
\item $F\cong \cO_C$ for a CM (Cohen-Macaulay)-curve $C$ on $Q$.
\item $F$ fits into a non-split short exact sequence
\[
\ses{\cO_C}{F}{\CC_p}
\]
for a CM-curve $C$ with $\chi(\cO_C(m))=4m$ and $p\in C$.
\end{enumerate}
\end{proposition}
\begin{proof}
Since $\chi(F)=1$, there exists a non-zero section $s: \cO_Q\lr  F$. Let $\cO_C=\mathrm{im}(s)$ be the image of the map $s$. By the stability of $F$, the curve $C$ is a connected CM-curve. If $\deg(C)\leq 3$, then by the degree-genus formula, we have $\chi(\cO_C)\geq 1$. This violates the stability of $F$. Hence $\deg(C)=4$. If the curve $C$ is planar, then $Q$ contains the plane $\langle C\rangle =\PP^2$ which is a contradiction. Hence by the degree-genus formula again, the possible arithmetic genera of $C$ are $p_a(C)=0$ or $1$. These two cases lead to the claim that we need.
\end{proof}
By Proposition \ref{clm4}, the moduli space $\bM_4$ decomposes into two parts:
\[
\bM_4=\bR_4\cup \bE_4
\]
where $\bR_4$ is the closure of the structure sheaves of rational CM-curves of degree $4$ and $\bE_4$ is the closure of the incidence variety of a (degenerated) elliptic quartic curve $C$ and a point $p\in C$. One can easily see that the two components have expected dimensions $12$ and $13$.
\subsection{The component $\bR_4$}
A general stable sheaf parameterized by $\bR_4$ has the following description.
\begin{proposition}[Proposition 10 \protect{\cite{Dro04}}]\label{acmpro}
Let $C\subset Q$ be a connected, non-degenerate CM-curve with Hilbert polynomial $\chi(\cO_C(m))=4m+1$. Then 
\begin{enumerate}
\item $C$ is an ACM (Arithmetically Cohen--Macaulay)-curve.
\item The ideal $I_{C/Q}$ is generated in degree two and it has the resolution of the form
\begin{equation}\label{eqres1}
\ses{2\cE(-2)}{5\cO_Q(-2)}{I_{C/Q}},
\end{equation}
where $\cE$ is the spinor bundle on $Q$.
\item The Zariski tangent space of $\bM_4$ at the point $[\cO_C]$ has dimension $12$.
\end{enumerate}
\end{proposition}
\begin{proof}
Item (1): The curve $C$ has arithmetic genus zero, and hence the structure sheaf $\cO_C$ is stable by \cite[Lemma 3.17]{CC10}. By Serre duality and the simplicity of stable sheaves
\cite[Proposition~1.1.10]{HL}, one can obtain
\[
\rH^1(\cO_C(n))=0 \; \text{for all}\; n\geq 0.
\]
Together with the non-degeneracy of the curve $C$, this implies that the curve $C$ is
ACM by the standard hyperplane-section argument. Item (2) comes from \cite[Proposition 10]{Dro04}. For item (3), applying $\Hom_Q(-,\mathcal O_C)$ to the resolution in \eqref{eqres1}
and using the vanishings $\rH^1(\mathcal O_C(k))=0$ for $k\geq 0$, one obtains $\dim \Hom_Q(I_{C/Q},\mathcal O_C)=12$. Since
\[
T_{[\mathcal O_C]}\bM_4
\cong \Ext_Q^1(\mathcal O_C,\mathcal O_C)
\cong \Hom_Q(I_{C/Q},\mathcal O_C)
\]
by $\rH^1(\cO_Q)=0$, the assertion follows.
\end{proof}
\begin{remark}
The resolution in \eqref{eqres1} of Proposition \ref{acmpro} comes from the Fourier-Mukai (FM) transform $\Phi$ among the space $\bM_1(Q)=\PP^3$ of lines on $Q$ and the quadratic threefold $Q$. Here we choose the universal sheaf $\cO_{\cC}$ as the FM-kernel of the transform $\Phi$. The FM-transform of the resolution \eqref{eqres1} becomes that of a twisted cubic curve with respect to the exceptional collection $D^b(\PP^3)=\langle \cO_{\PP^3}, T_{\PP^3}(-1), \cO_{\PP^3}(1), \cO_{\PP^3}(2)\rangle$.
\end{remark}
\begin{corollary}\label{syre}
In the setting of Proposition~\ref{acmpro}, the space of
linear syzygies among the five quadratic generators of $I_{C/Q}$ is
eight-dimensional.
\end{corollary}

\begin{proof}
Twisting the resolution \eqref{eqres1} by the line bundle $\cO_Q(3)$ and then taking
global sections, we obtain the exact sequence
\begin{equation}\label{eq:mul}
0\lr 2\rH^0(\cE(1))\lr5\rH^0(\cO_Q(1))\cong \rH^0(I_{C/Q}(2))\otimes \rH^0(\cO_Q(1))\stackrel{m}{\lr} \rH^0(I_{C/Q}(3))\lr 2\rH^1(\cE(1))=0.
\end{equation}
Here the map $m$ is the natural multiplication map. Therefore, the space of linear syzygies is
\[
\ker(m)= 2\rH^0(\cE(1)),
\] which is $8$-dimensional by the exact sequence \eqref{stre}.
\end{proof}
Clearly the degenerated, CM-rational curves of degree $4$ are contained in quadric surfaces of rank $\geq 3$.
\subsection{The component $\bE_4$}
We describe the so-called \emph{elliptic component} $\bE_4$.
\begin{proposition}\label{cmell}
Let $\pi:\cS\to (\mathbb P^4)^\vee=B$ be the universal hyperplane
bundle of $Q$. Then the Hilbert scheme $\bH_{4m}(Q)$ of CM-curves $C\subset Q$ with
Hilbert polynomial $\chi(\mathcal O_C(m))=4m$ is isomorphic to
\[
\bH_{4m}(Q)\cong \mathbb P_{(\mathbb P^4)^\vee}
\bigl(\pi_*\mathcal O_{\cS}(2)\bigr).
\]
In particular, it is a $\mathbb P^8$-bundle over
$B$.
\end{proposition}
\begin{proof}
It is well-known that the curve $C$ is contained in a hyperplane
$H (\cong\mathbb P^3)$. This hyperplane is unique, since otherwise $C$
would be contained in a plane, whereas a plane quartic has arithmetic genus $3$. Let $S_H=Q\cap H$. Then $S_H$ is a quadric surface of rank $\geq 3$ and thus 
\[
C\in |\mathcal O_{S_H}(2)|
\]
by \cite[Theorem 5.2]{AV92}. Conversely, every divisor in $|\mathcal O_{S_H}(2)|$ is a complete
intersection of two quadrics in $H\cong\mathbb P^3$, hence is a
CM-curve with Hilbert polynomial $4m$. Therefore the curves under
consideration are parametrized by the relative linear system
\[
\PP_B\bigl(\pi_*\cO_{\cS}(2)\bigr).
\]
For every $[H]\in B=(\mathbb P^4)^\vee$, the exact sequence
\[
\ses{\cO_{\mathbb P^3}}{\cO_{\mathbb P^3}(2)}{\cO_{S_H}(2)}
\]
gives $h^0\bigl(\mathcal O_{S_H}(2)\bigr)=9$ and $h^1\bigl(\mathcal O_{S_H}(2)\bigr)=0$. Hence, by cohomology and base change, the direct image sheaf $\pi_*\mathcal O_{\mathcal S}(2)$ is a locally free sheaf of rank $9$ on $B$. Consequently, we proved the claim.
\end{proof}
\begin{remark}\label{remcm}
From the proof of Proposition \ref{cmell}, every CM-curve $[C]\in \bH_{4m}(Q)$ has the following Koszul resolution:
\[
0\lr \cO_Q(-3)\stackrel{[-\mathbf{h},\mathbf{q}]^T}{\lr}  \cO_Q(-2)\oplus \cO_Q(-1)\stackrel{[\mathbf{q}, \mathbf{h}]}{\lr} \cO_Q\lr \cO_C\lr 0,
\]
where $C$ is defined by $I_{C/Q}=\langle \mathbf{h}, \mathbf{q}\rangle$ for $\deg(\mathbf{h})=1$, $\deg(\mathbf{q})=2$.
\end{remark}

\begin{proposition}\label{4mpro}
Let $[E]\in \bH_{4m}(Q)$ with a point $p\in E$. Then
\begin{enumerate}
\item $\Ext^1(\CC_p, \cO_E)\cong \CC$ and
\item the unique non-split extension
\[
\ses{\cO_E}{F}{\CC_p}
\]
is the dual $F= \cExt_Q^2(I_{p/E},\omega_{Q})$ of the ideal sheaf $I_{p/E}$, which is a stable sheaf.
\item The component $\bE_4$ is isomorphic to the universal curve $\cC$ over the Hilbert scheme $\bH_{4m}(Q)$ of CM-curves, where $\cC$ is a $\PP^7$-bundle over $\bB=\{(p,H)\in Q\times(\PP^4)^\vee \mid p\in H\}$.
\end{enumerate}
\end{proposition}
\begin{proof}
By Proposition \ref{cmell}, $E$ is a complete intersection of type $(1,2)$ in
$Q$. In particular, $E$ is a Gorenstein curve of arithmetic genus one, and  $\omega_E \cong \mathcal O_E$.

For (1), by Serre duality and the exact sequence $\ses{I_{E/Q}}{\cO_Q}{\cO_E}$, we have
\[
\Ext^1(\CC_p, \cO_E)=\Ext^2(\cO_E, \CC_p)=\Ext_Q^1(I_{E/Q},\mathbb C_p)^\vee.
\]
Applying $\Hom_Q(-,\mathbb C_p)$ to the Koszul resolution in Remark \ref{remcm} provides
\[
\CC^2\stackrel{(-\mathbf{h}(p),\,\mathbf{q}(p))}{\longrightarrow}\CC\lr\Ext_Q^1(I_{E/Q},\CC_p)\lr 0.
\]
Since $p\in E$, we have $\mathbf{h}(p)=\mathbf{q}(p)=0$ and thus
\[
\Ext_Q^1(\mathbb C_p,\mathcal O_E)
\cong
\Ext_Q^1(I_{E/Q},\mathbb C_p)^*
\cong\mathbb C.
\]
For (2), let us apply $\mathcal Hom_Q( - ,\omega_Q)$ to the exact sequence
\[
\ses{I_{p/E}}{\cO_E}{\CC_p}.
\]
By $\mathcal Ext_Q^2(\cO_E,\omega_Q)\cong\omega_E=\cO_E$, $\mathcal Ext_Q^{<3}(\mathbb C_p,\omega_Q)=0$ and $\mathcal Ext_Q^3(\mathbb C_p,\omega_Q)\cong\mathbb C_p$, 
we obtain an exact sequence
\begin{equation}\label{eq4.3}
\ses{\cO_E}{\cE xt_Q^2(I_{p/E},\omega_Q)}{\CC_p}.
\end{equation}
    The middle term in \eqref{eq4.3} is a pure one-dimensional CM-sheaf, so the extension cannot split. By the uniqueness by item (1), we proved the first part. It remains to prove that $F$ is stable. By Serre duality, one can easily see that $h^0(F)=1$. Hence every nonzero section of $F$ is a scalar multiple of the canonical
section $\mathcal O_E\hookrightarrow F$, whose image has degree $4$. Suppose that $0\neq G\subsetneq F$ is a destabilizing pure
one-dimensional subsheaf, and let $\chi(G(m))=dm+\chi(G)$. Then
\begin{equation}\label{eq:14}
\frac{\chi(G)}{d}\geq \frac14.
\end{equation}
If $d<4$, then $\chi(G)\geq1$ and thus $h^0(G)\neq0$. The inclusion $G\hookrightarrow F$ would then give a
nonzero section of $F$ whose image has degree at most $d<4$, contradicting
the uniqueness of the section above. If $d=4$, then $F/G$ is a nonzero zero-dimensional sheaf, and hence
\[
\chi(G)=\chi(F)-\ell(F/G)<1.
\]
Thus $\chi(G)\leq0$, which is a contradiction to the inequality \eqref{eq:14}. Therefore no proper subsheaf destabilizes $F$ and hence $F$ is stable.

The first part of item (3) directly comes from (1) and (2). Clearly the projection $\bB\to Q$ is a $\PP^3$-bundle over $Q$. The evaluation at the
marked point in $Q$ gives a surjection
\[
\mathcal V_\bB\lr \mathcal L,
\]
where $\mathcal V_\bB$ is the pullback of the rank $9$ bundle
parametrizing $|\mathcal O_{Q\cap H}(2)|$.  Its kernel $\mathcal K$
is locally free of rank $8$, and hence $\mathcal C\cong\PP_\bB(\mathcal K)$.
\end{proof}

\section{Computation of DT-type invariants}
For a CY $4$-fold $Y$, let $M=M_{\beta}(Y)$ be the moduli space of stable sheaves $F$ with $[F]=\beta$ and $\chi(F)=1$. Consider a normalized universal sheaf 
\begin{equation}\label{norf}
\cF_{\mathrm{norm}}=\cF\otimes \pi_M^*(\det (R\pi_{M*}\cF))^{-1}
\end{equation}
of the universal sheaf $\cF$ over the product space $Y \times M$. For an integral class $\gamma\in \rH^{k}(Y,\ZZ)$, let
\begin{align*}
\tau_i(\gamma) := (\pi_M)_*( \pi_Y^*\gamma  \cdot\mathrm{ch}_{3+i}(\cF_{\text{norm}}) ) \in \rH^{k+2(i-1)}(M)
\end{align*}
be the \emph{insertion} of $\gamma$ with respect to the normalized universal sheaf $\cF_{\text{norm}}\in \mathrm{Coh}(Y\times M)$.
For a suitable choice of orientation on the moduli space $M$, we define
the invariant associated with the insertion $\tau_i(\gamma)$ by
\begin{equation}\label{eq-inin}
\left\langle \tau_i(\gamma) \right\rangle_{\beta}
:=
\int_{[M]^{\mathrm{vir}}} \tau_i(\gamma).
\end{equation}
When $i=0$, we call $\tau_0(\gamma)$ a \emph{primary insertion} and
$\langle \tau_0(\gamma)\rangle_{\beta}$ a \emph{primary invariant} (or \emph{Donaldson-Thomas invariant}).
When $i= 1$, we call $\tau_i(\gamma)$ a \emph{descendant insertion}
and $\langle \tau_i(\gamma)\rangle_{\beta}$ a \emph{descendant invariant}.

\subsection{DT-type invariants on $Q$ via localization}
Let $Y = \mathrm{Tot}(K_Q)$ be the total space of the canonical line bundle $K_Q=-3H$ with the canonical projection map $\pi : Y \to Q$. By \cite[Theorem 6.5]{CL14}, \[\bM_{\beta}(Q)\cong \bM_{i_* \beta}(Y), \quad F\mapsto i_*F\] for the zero section $i:Q\hookrightarrow Y=\mathrm{Tot}(K_Q)$. Also, by the Grothendieck-Riemann-Roch theorem, for $h_{2-i}\in \rH^{4-2i}(Q,\mathbb Z)$ the insertion becomes
\[
\tau_i(h_{2-i})
=
\pi_{\bM *}\left(
\pi_Q^*h_{2-i}\cdot
\left\{
\mathrm{ch}(\cF_{\mathrm{norm}})
\cdot \mathrm{td}(K_Q)^{-1}
\right\}_{i+2}
\right),
\]
where $\{\cdot\}_{i+2}$ denotes the component of total cohomological degree $2i+4$ on $Q\times \bM$.
Thus the integration in \eqref{eq-inin} is defined against the virtual class $[\bM_{\beta}(Q)]^{\mathrm{vir}}\in \rH_2(\bM_{\beta}(Q), \QQ)$ with a suitable orientation $(-1)^{c_1(Q)\cdot \beta-1}=(-1)^{3d-1}$ where $\beta = d\ell$ for the line class $\ell$ (\cite[Corollary 3.39]{Tho00}). In this setting, we formulate a localization formula for integrals with insertions.
Let
\[
\bM_\beta^T=\coprod_\alpha \bZ_\alpha
\]
be the decomposition of the fixed locus into connected components. For simplicity, we assume that each $\bZ_\alpha$ is isomorphic to $\PP^k$ for $k=0$ or $1$. For a connected component $\bZ:=\bZ_\alpha\subset \mathbf{M}_\beta^T$, let $\pi_{\bZ}:Q\times \bZ\lr \bZ$ be the projection map and $\mathcal F_\bZ$ denote the restriction of the normalized, universal family $\cF$ to $Q\times \bZ$. We define the restriction of the virtual tangent space of $\bM_\beta$ as
\begin{equation}\label{reclas}
K_\bZ=T_{\mathbf M_\beta}^{\vir}|_\bZ=
\mathcal [\cO_\bZ]-
R\pi_{\bZ*}R\mathcal Hom(\mathcal F_\bZ,\mathcal F_\bZ)
\in K_T(\bZ).
\end{equation}
We decompose $K_\bZ$ into its fixed and moving parts:
\[
K_\bZ=K_\bZ^{\fix}+K_\bZ^{\mov}.
\]
From $K_\bZ^{\fix}=T_\bZ$, $N_\bZ^{\vir}=K_\bZ^{\mov}$ and Theorem \ref{loccomp}, we have
\begin{equation}\label{lcz}
\langle\tau_i(h_{2-i})\rangle_\beta
=
(-1)^{3d-1}\cdot\sum_{\bZ\subset\mathbf M_\beta^T}
\operatorname{Cont}_\bZ,
\end{equation}
where the contributions of $\bZ$ in \eqref{lcz} are
\begin{equation}\label{conz}
\operatorname{Cont}_\bZ=\int_{[\bZ]^{\mathrm{vir}}}\frac{\tau_i(h_{2-i})|_\bZ}
     {e(K_\bZ^{\mov})}=
\begin{cases}
\displaystyle
\tau_i(h_{2-i})|_{[F]}\,\cdot
e\bigl(\chi(F,F)-1\bigr),
& \bZ=\{[F]\},\\[4pt]
\displaystyle
\int_{[\bZ]}
\tau_i(h_{2-i})|_\bZ\,\cdot
e\bigl(
\chi_\bZ(\mathcal F_\bZ,\mathcal F_\bZ)
-[\mathcal O_\bZ]+[\mathcal O_\bZ(2)]
\bigr),
& \bZ\simeq\mathbb P^1.
\end{cases}
\end{equation}
Here the restricted insertion $\tau_i(h_{2-i})|_\bZ \in \rH_T^{2}(\bZ)$ can be computed from the equivariant universal sheaf on $Q\times \bZ$. The remaining part of this paper is to find the $T$-fixed locus of the moduli space $\bM_4$ and compute the contribution in  \eqref{conz}.
\subsection{$T$-fixed sheaves in $\bR_4$}\label{sub:rfixed}
In this subsection, we are willing to classify all torus-fixed stable sheaves on $Q$.

Note that if \(Y\subset X\) is the closed subscheme defined by the ideal
sheaf \(I_Y\), then every coherent sheaf \(\mathcal F\) on \(X\) satisfying
\(I_Y\mathcal F=0\) is naturally an \(\mathcal O_Y\)-module. In particular,
\[
\cF \cong  \cF|_Y.
\]
The following is a classification of all $T$-fixed, CM-rational quartic curves on $Q$. Throughout this subsection, $d=\deg(C_{\mathrm{red}})$ denotes the degree of the reduced scheme $C_{\mathrm{red}}$ of the CM-curve $C$.
\subsubsection{Degree $d=1$ case}\label{subsub51}
\begin{proposition}\label{mult4line}
A $T$-invariant CM-curve $C$ in $Q$ supported on the $T$-fixed line
$L=Z(z_2,z_3,z_4)$ is unique, which is defined by
\[
I_C=\langle z_2z_3,z_2z_4,z_3^2,z_3z_4,z_4^2\rangle.
\]
\end{proposition}

\begin{proof}
By the general argument of \cite[Section 4]{BF87}, the curve $C$
must have generic embedding dimension $2$. Also, there exists a
CM-filtration of $C$:
\[
L\subset D\subset W\subset C
\]
where Hilbert polynomials are $(D,W)=(2m+1,3m)$ or $(2m+1,3m+1)$
(\cite[Section 2]{NS03}). But the former case is impossible since $Q$
does not contain any plane. For the latter case, there is a $T$-fixed
point in the nested Hilbert scheme of $C$ with the prescribed Hilbert
polynomials by the Borel fixed theorem. Since all of the curves $D$ and $W$
are CM from the global description of the Hilbert scheme (or direct
computation), one can choose two types of $T$-invariant CM-subcurves
as follows.
\[
D=Z(z_3,z_4),\quad
W=Z(z_4,z_2z_3,z_3^2)
\;\mathrm{or}\;
W=Z(z_3,z_2z_4,z_4^2).
\]
From the structure sequence
$\ses{I_{W/C}=\cO_L(-1)}{\cO_C}{\cO_W}$, we need to find a 
$T$-equivariant surjection
$\phi:I_W\twoheadrightarrow\cO_L(-1)$. Since the image
$\phi(I_LI_W)=I_L\phi(I_W)$ is an
$\cO_L=\cO_Q/I_L$-module, we have $\phi(I_LI_W)=0$. Thus the
surjection map $\phi$ factors through
\[
\bar{\phi}:I_W/I_LI_W=\cF\twoheadrightarrow\cO_L(-1).
\]
But the quotient sheaf $\cF$ is presented by
\[
\cF=\langle \bar z_4,\bar z_2\bar z_3,\bar z_3^2\rangle
\]
with a relation
$z_1\bar z_3^2=0$ obtained from the defining equation of the
quadratic threefold $Q$. Clearly $\cF$ is an $\cO_L$-module.
By considering the degree of the generators, the map $\bar\phi$
is of the form
\[
\bar\phi(\bar z_4)=au,\quad
\bar\phi(\bar z_2\bar z_3)=l_1u,\quad
\bar\phi(\bar z_3^2)=l_2u
\]
for some $a\in\CC^*$, $l_i\in\rH^0(\cO_L(1))$ and the local
generator $u$ of the line bundle $\cO_L(-1)$. But the equality
\[
0=\bar\phi(z_1\bar z_3^2)
=z_1\bar\phi(\bar z_3^2)
=z_1l_2u
\]
implies $l_2=0$. If $a=0$ and $l_1\neq0$, then $\bar\phi$ is not
surjective, which contradicts the assumption. If $a\neq0$, then
$l_1=0$ by considering the weight vectors of generators. By lifting
the kernel $\ker(\bar{\phi})
=\langle\bar z_2\bar z_3,\bar z_3^2\rangle$ along the quotient map, one can easily check that
\[
\ker(\phi)
=I_LI_W+\langle z_2z_3,z_3^2\rangle
=\langle z_2z_3,z_2z_4,z_3^2,z_3z_4,z_4^2\rangle.
\]
One can easily see that the other one also gives the same ideal in
the claim.
\end{proof}
\subsubsection{Degree $d=2$ case}\label{subsub52}
For stable sheaves supported on two fixed lines meeting at a point, the possible double structure can be classified in a straightforward manner via Ferrand doubling (\cite{BF87}).
\begin{proposition}\label{22l}
Let\[
Y=L_{z_0z_1}\cup L_{z_0z_3}=Z(z_2,z_4).
\]
Then the $T$-fixed curves $C$ with $[C]=2[L_{z_0z_1}]+2[L_{z_0z_3}]$ with
Hilbert polynomial $\chi(\cO_C(m))=4m+1$ are precisely those defined by
\[
I_{C}
=
\langle z_2^2,z_2z_4,z_4^2,z_1z_2,z_3z_4\rangle \;\mathrm{or}\; \langle z_2^2,z_2z_4,z_4^2,z_2z_3,z_1z_4\rangle.
\]
\end{proposition}

\begin{proof}
Let $J=I_Y=\langle z_2,z_4\rangle$. Since \(Y\subset Q\) is a complete intersection, a
Ferrand double structure on \(Y\) is determined by a \(T\)-equivariant
quotient
\[
\phi: J\twoheadrightarrow \cM
\]
such that a torsion free sheaf $\cM$ (coming from the pureness of the structure sheaf $\cO_C$) has Hilbert polynomial $\chi(\cM(m))=2m$ and is supported on $Y$. Since $\cM$ is an $\cO_Y$-module, the map $\phi$ factors through $\bar{\phi}: J/J^2\twoheadrightarrow \cM$. Write $L_1=L_{z_0z_1}$ and $L_2=L_{z_0z_3}$.
If $\cM$ is a locally free sheaf, then it fits into the $T$-equivariant Mayer–Vietoris sequence
\[
\ses{\cM}{\cO_{L_1}(a)\oplus \cO_{L_2}(b)}{\CC_p}.
\]
From $J/J^2=\cO_Y(-1)\oplus \cO_Y(-1)$, the Hilbert polynomial of $\cM$, and equivariant surjectivity of the map $\phi$, one can see that $(a,b)=(-1, 0)$ or $(0, -1)$. Since the other case is symmetric, we only consider the case $(a,b)=(0,-1)$. Then the equivariant sequence induces that the section space $\rH^0(\cM(1))$ is generated by $\langle (z_0, 1), (z_1,0)\rangle $ where the weight difference is $\wt (z_0, 1)-\wt(z_1,0)=(0,2)$. But the weight difference of the conormal sheaf $J/J^2$ is $\wt(z_2)-\wt(z_4)=(1,1)$. These are not equal up to the sign and twisting. Hence no T-equivariant surjective morphism $\phi$ exists. Now, let $\cM$ be a non--locally free sheaf. 
Then it is well-known that the sheaf $\cM$ is of the form
\[
\cM=
\nu_*
\bigl(
\mathcal O_{L_1}(a)\oplus
\mathcal O_{L_2}(b)
\bigr)
\]
for the normalization map $\nu:L_1\sqcup L_2\longrightarrow Y$ of the curve $Y$. Clearly, the direct sum is a $T$-eigenspace decomposition. The surjectivity of the map $\phi$ implies that $a=b=-1$. Hence, any $T$-equivariant map $\bar\phi$ is a diagonal (or off-diagonal) matrix since the two generators of $J/J^2$ have distinct characters. Therefore,
\[
\bar\phi(\bar z_2)=u_1, \bar\phi(\bar z_4)=u_2, \;\mathrm{or}\; \bar\phi(\bar z_2)=u_2,\bar\phi(\bar z_4)=u_1\]
where $u_i$ denotes a local $T$-invariant generator of the equivariant line bundle $\cO_{L_i}(-1)$.
In the first case the kernel of the map $\bar{\phi}$ is $\ker(\bar{\phi} )=\langle z_3\bar z_2, z_1\bar z_4\rangle$.
Hence one can easily see that
\[
\ker(\phi)=
J^2+\langle z_2z_3, z_1z_4\rangle
=
\langle z_2^2,z_2z_4,z_4^2,z_2z_3,z_1z_4\rangle.
\]
Interchanging the two generators $u_1$, $u_2$ gives the other ideal in the claim.
\end{proof}
\begin{remark}
Note that if $C$ is supported on a smooth conic in $Q$, then weight considerations show that no $T$-invariant curve $C$ exists. Since the detail is similar to the proof of Proposition \ref{22l}, we skip the detail.
\end{remark}
Now, we find a torus fixed curve $C$ of type $[C]=3[L]+[L']$, where the lines $L=Z(z_2, z_3, z_4)$ and $L'=Z(z_1, z_2, z_4)$ meet at a point $p$. One can easily check that the CM-curve $C$ may contain a CM-subcurve $C_2$ with $[C_2]=2[L]$ and $\chi(\cO_{C_2}(m))=2m+1$ or $2m+2$ (\cite[Case 3 of Proposition 5.7] {CKK26} and \cite[Section 2]{Nol97}). The existence of a $T$-invariant CM-subcurve $C_2$ comes from the Borel fixed theorem. For the former case, it is clearly a CM-curve. For the latter case, the $T$-invariant CM-subcurve $C_2$ may be a non-CM curve. In this case, consider the maximal CM-subcurve of $C_2$. Because of its uniqueness, it must be a $T$-invariant one. Also, by stability of the structure sheaf $\cO_C$, its Hilbert polynomial must be $2m+1$. Thus this case reduces to the first one. Also by Ferrand doubling (\cite[Section 2]{BF87}), one can easily find the $T$-invariant CM-curves $C_2$ defined by
\[
C_2=V( z_3, z_4) \;\mathrm{or}\; V(z_2, z_3^2, z_3z_4, z_4^2)
\]
\begin{proposition}\label{onef}
Under the above setting, there are two types of $T$-invariant curves $C$ with $[C]=3[L]+[L']$. That is, there exists a one-parameter family of CM-curves:
\begin{equation}\label{fam1}
I_{C_{t}}=\left\langle z_2z_3,\,z_2z_4,\,z_3z_4,\,z_4^{2},\,z_2^{2}-t z_0z_4 \right\rangle, \;t\in\mathbb A^{1}.
\end{equation}
Another one is given by
\[
I_C
=\langle z_2,\,z_4^{3},\,z_3z_4^{2},\,z_3^{2}z_4\rangle.
\]
\end{proposition}
\begin{proof}
We need to classify the surjection $\phi:I_{C_2}\twoheadrightarrow \cM$ where the torsion free sheaf $\cM$ is supported on $Y=L\cup L'=V(z_2, z_4)$ with Hilbert polynomial $\chi(\cM(m))=2m$ or $2m-1$. 

Case 1) ($\chi(\cM(m))=2m$): Since $I_Y\mathrm{\cM}=0$ by the torsion-freeness of $\cM$, the sheaf $\cM$ is an $\cO_Y$-module and thus the surjection factors through the quotient $\bar\phi: F=I_{C_2}/I_YI_{C_2}\twoheadrightarrow \cM$. Let us firstly assume that $\cM$ is locally free on $Y$. From $I_{C_2}|_{L'}=\langle z_3\rangle=\cO_{L'}(-1)$ up to torsion, we have an exact sequence
\[\ses{\cM}{\cO_{L}\oplus \cO_{L'}(-1)}{\CC_p}.\]
Then $\rH^0(\cM(1))=\langle s_0=(z_0, 1), s_1=(z_1,0)\rangle$ by the exact sequence. The sheaf $F$ on $Y$ is generated by $\bar{z_3}$ and $\bar z_4$ which are regarded as local generators of $\cO_Y(-1)$. Using the
identification
\[
\operatorname{Hom}_{\mathcal O_Y}
\bigl(\mathcal O_Y(-1),\mathcal M\bigr)
=
H^0\bigl(\mathcal M(1)\bigr),
\]
the map $\bar\phi$ is determined by a constant matrix. By weight comparison among $\wt(\bar z_3)-\wt( \bar z_4)=\wt(s_0)-\wt(s_1)=(0,2)$ and the relation $z_3\bar z_4=0$ on the sheaf $F$, the map $\bar \phi $ is given by
\[
\bar\phi(\bar z_3)=a s_0,\; \bar\phi(\bar z_4)=b s_1,\; ab\neq0.
\]
Now $z_1s_0=z_1(z_0,1)=(z_0z_1, 0)$ since $z_1|_{L'}=0$. Therefore $z_1s_0=z_0s_1$. Consequently,
\[
\bar \phi (b z_1\bar{z_3}-a z_0 \bar{z_4})=b z_1\bar{\phi}(\bar{z_3})-a z_0\bar\phi (\bar{z_4})=ab(z_1s_0-z_0s_1)=0.
\]
That is, $\langle b z_1\bar{z_3}-a z_0 \bar{z_4}\rangle \subset \ker(\bar\phi)$. By lifting this to the ideal $I_{C_2}$ along the canonical map, we obtain that
\[
I_YI_{C_2}+\langle b z_1z_3-a z_0z_4\rangle\subset \ker(\phi).
\]
By comparing the Hilbert polynomial of both ones, the equality holds and thus we obtain the claimed family for $t=-\frac{a+b}{b}\in \AA^1\setminus\{-1\}$.

Secondly, let $\cM$ be a non-locally free sheaf.  The quotient sheaf $\cM$ is $\cM=v_*(\cO_L(-1)\oplus \cO_{L'}(-1))$ by the similar argument in the previous Proposition. By twisting the line bundle $\cO(1)$ in the map $\bar\phi$, we have a degree-zero $T$-equivariant surjective map:
\[
\bar\phi\otimes \mathrm{id}: (I_{C_2}/I_YI_{C_2})(1)\twoheadrightarrow \cM(1).
\]
The generators of the former have different weights, and thus the map is a constant diagonal (or off-diagonal) matrix. Using the relation $z_3\bar{z_4}=0$ on the quotient space $F$, one can easily see that the map is just the identity map. From this one, by the same argument as before, one can see that
\[
\ker(\phi)=I_YI_{C_2}+\langle z_1z_3\rangle
\]
which is a special point $t=-1$ of the one-parameter family of the claimed one.

Case 2) $\chi(\cM(m))=2m-1$: By the same argument as in Case 1, we obtain the $T$-fixed ideal as in the claim. In fact, there are two $T$-fixed curves. One of them is a member of the one-parameter family of Case 1.
\end{proof}
\begin{remark}\label{re:fla}
By extending the family \eqref{fam1} to $t=\infty$, we have a unique non-CM curve
\[
I_{C_{\infty}/Q}
=
\langle z_2z_3,z_2z_4,z_3z_4,z_4^2, z_0z_4\rangle.
\]
A direct computation in \texttt{Macaulay2} \cite{M2} shows that
$C_\infty$ has a unique embedded point at
$p_{z_1}=[0:1:0:0:0]$.
After taking an elementary
modification at this point, the family extends to a family of stable
sheaves over $\PP^1$, whose special fiber belongs to $\bE_4$.
Thus, we obtain a connected component
$\bZ\cong\PP^1$ of $\bM_4^T$.
The localization contribution of this case is computed in Section~\ref{sssecp1}.
\end{remark}

\subsubsection{Degree $d\geq3$ case}\label{subsub53}
\begin{proposition}
Let $L_1=Z(z_0,z_2,z_3)$, $L_2=Z(z_2,z_3,z_4)$ and $L_3=Z(z_1,z_2,z_4)$.
Let $[C]=2[L_1]+[L_2]+[L_3]$. Then the possible structure sheaves $\cO_C$ are given by
\[
I_C=\langle z_0z_2, z_2^2, z_2z_3, z_1z_3, z_3z_4\rangle \;\mathrm{or}\;\langle z_2, z_0z_1z_3, z_1z_3^2, z_3^2z_4\rangle.
\]
\end{proposition}
\begin{proof}
The reduced curve \(Y=L_1\cup L_2\cup L_3= Z(z_2, z_1z_3, z_3z_4)\) is a chain of lines $L_i$, which is a locally complete intersection. Let us find a $T$-equivariant surjection $\bar\phi: I_Y/I_{L_1}I_Y\twoheadrightarrow \cO_{L_1}(-1)$. By direct computation, we have
\[
I_Y/I_{L_1}I_Y
=\langle \bar z_2, \bar z_3\rangle.
\]
Since the two generators $\bar z_2$ and $\bar z_3$ have distinct weight vectors and the source and target line bundles have the same degree, every $T$-equivariant morphism is represented by
\[
\bar{\phi}(\bar z_2)=u,\; \bar{\phi}(\bar z_3)=0,
\]
or
\[
\bar{\phi}(\bar z_2)=0,\; \bar{\phi}(\bar z_3)=u,
\]
where $u$ is a $T$-equivariant local generator of $\cO_{L_1}(-1)$. Therefore, their kernels lift to the ideals
\[
I_{C}=I_{L_1}I_Y+\langle z_1z_3, z_3z_4\rangle, \;\text{or}\;I_{L_1}I_Y+\langle z_2\rangle.
\]
Switching $L_1$ with $L_3$, we have the second ideal. For the $L_2$ case, the same computation provides the result.
\end{proof}
\begin{lemma}
Let $C$ be a $T$-fixed CM-curve
such that $[C]=s[Q]+t_1[L_1]+t_2[L_2]$, $2s+t_1+t_2=4, s\geq 1$ for a $T$-fixed conic $Q$ and the lines $L_i$. Then $s=t_1=t_2=1$. In this case, the $T$-invariant curves $C$ are given by
\[
I_C=\langle z_0z_1, z_1z_2, z_1z_3, z_2z_3, z_3z_4\rangle,\;\mathrm{or}\;\langle z_0z_1, z_0z_3, z_1z_2, z_1z_3, z_2z_3\rangle.
\]
\end{lemma}
\begin{proof}
If $s=2$, then $t_1=t_2=0$. We claim that there do not exist $T$-fixed double conics supported on $Q$ in $Q_3$.
Let $Q=V(z_1, z_3)$.
It is enough to find a $T$-equivariant surjective morphism
\[
\bar\phi: I_Q/I_Q^2=\langle \bar z_1,\bar z_3\rangle\twoheadrightarrow \cL
\]
for a $T$-equivariant line bundle $\cL$ of degree $-1$ on $Q$. Then $\bar\phi(\bar z_1)=s_1$, $\bar\phi(\bar z_3)=s_2$ where the eigenvectors $s_i\in \rH^0(\cL(1))$ via the $T$-action are linearly independent (if not, $\bar\phi$ is not surjective). If $\cL=\cO_Q(-p)$ for a fixed point $p$, then $\rH^0(\cL(1))=\langle z_2, z_4\rangle$. Hence $\wt(s_1)-\wt(s_2)=\pm(1,1)$. For the other fixed point, we have the same difference. But $\wt(\bar z_1)- \wt(\bar z_3)=\pm(2,-2)$. So there does not exist any $T$-equivariant surjection map $\bar \phi$.

Let $s=1$, $t_1=2$. Note that the supporting curve $Y=Q\cup L_1$ is a locally complete intersection. By a local computation as in Proposition \ref{22l}, one can show that there does not exist any $T$-invariant curve.

Let $s=t_1=t_2=1$. Excluding the arithmetic genus $p_a(C)=1$, there are exactly two types of $T$-invariant rational curves: Tree and concurrent triple-point configuration as in the claim.
\end{proof}
We have now classified all possible representation types of $T$-fixed rational curves on $Q$ through Sections \ref{subsub51}--\ref{subsub53}; see Table~\ref{tab:traq}.
\begin{table}[ht]
\centering
\renewcommand{\arraystretch}{2.2}
\setlength{\tabcolsep}{10pt}
\begin{tabular}{|c|c|c|c|}
\hline
\textbf{No.} & \textbf{Graph} & \textbf{Numbers}
& \textbf{Representative ideals $I_{C/Q}$}\\
\hline
\hline

$1$
&
\begin{tikzpicture}[baseline=-0.5ex]
\fill (0,0) circle (2pt);
\fill (1,0) circle (2pt);
\draw (0,0)--node[above] {{\scriptsize$4$}}(1,0);
\end{tikzpicture}
&
$4\times 1$
&
$\langle z_2z_3,z_2z_4,z_3^2,z_3z_4,z_4^2\rangle$
\\
\hline
\hline

$2$
&
\multirow{2}{*}{
\begin{tikzpicture}[baseline=-0.5ex]
\fill (1,1) circle (2pt);
\fill (1,0) circle (2pt);
\fill (0,1) circle (2pt);
\draw (0,1)--node[above] {{\scriptsize$2$}}
(1,1)--node[right] {{\scriptsize$2$}}(1,0);
\end{tikzpicture}}
&
$4\times1$
&
$\langle z_2^2,z_2z_4,z_4^2,z_1z_2,z_3z_4\rangle$
\\
\cline{1-1}\cline{3-4}

$3$
&
&
$4\times1$
&
$\langle z_2^2,z_2z_4,z_4^2,z_2z_3,z_1z_4\rangle$
\\
\hline

$4$
&
\multirow{2}{*}{
\begin{tikzpicture}[baseline=-0.5ex]
\fill (1,1) circle (2pt);
\fill (1,0) circle (2pt);
\fill (0,1) circle (2pt);
\draw (0,1)--node[above] {{\scriptsize$1$}}
(1,1)--node[right] {{\scriptsize$3$}}(1,0);
\end{tikzpicture}}
&
$8\times1$
&
$\langle z_2,z_4^3,z_3z_4^2,z_3^2z_4\rangle$
\\
\cline{1-1}\cline{3-4}

$5$
&
&
$4\times2\times1$
&
$\left\langle
z_2z_3,z_2z_4,z_3z_4,z_4^2,z_2^2-t z_0z_4
\right\rangle,
\; t \in\mathbb A^1$
\\
\hline
\hline

$6$
&
\multirow{2}{*}{
\begin{tikzpicture}[baseline=-0.5ex]
\fill (0,0) circle (2pt);
\fill (0,1) circle (2pt);
\fill (1,1) circle (2pt);
\fill (1,0) circle (2pt);
\draw (0,0)--node[left] {{\scriptsize$2$}}
(0,1)--node[above] {{\scriptsize$1$}}
(1,1)--node[right] {{\scriptsize$1$}}(1,0);
\end{tikzpicture}}
&
$8\times1$
&
$\langle z_0z_2,z_2^2,z_2z_3,z_1z_3,z_3z_4\rangle$
\\
\cline{1-1}\cline{3-4}

$7$
&
&
$8\times1$
&
$\langle z_2,z_0z_1z_3,z_1z_3^2,z_3^2z_4\rangle$
\\
\hline

$8$
&
\begin{tikzpicture}[baseline=-0.5ex]
\fill (0,0) circle (2pt);
\fill (0,1) circle (2pt);
\fill (1,1) circle (2pt);
\fill (1,0) circle (2pt);
\draw (0,0)--node[left] {{\scriptsize$1$}}
(0,1)--node[above] {{\scriptsize$2$}}
(1,1)--node[right] {{\scriptsize$1$}}(1,0);
\end{tikzpicture}
&
$4\times1$
&
$\langle z_2^2,z_2z_3,z_2z_4,z_1z_3,z_3z_4\rangle$
\\
\hline
\hline

$9$
&
\begin{tikzpicture}[baseline=1.0ex]
\fill (0,0) circle (2pt);
\fill (1,1) circle (2pt);
\fill (0,1) circle (2pt);
\fill (1,0) circle (2pt);
\draw (1,1)--node[right] {{\scriptsize$1$}}(1,0);
\draw (0,1)--node[above] {{\scriptsize$1$}}(1,1);
\draw[line width=1.5pt,gray]
(0,0)--node[left] {{\scriptsize$2$}}(1,1);
\end{tikzpicture}
&
$4\times1$
&
$\langle z_0z_1,z_1z_2,z_1z_3,z_2z_3,z_3z_4\rangle$
\\
\hline

$10$
&
\begin{tikzpicture}[baseline=-0.5ex]
\fill (0,0) circle (2pt);
\fill (1,1) circle (2pt);
\fill (0,1) circle (2pt);
\fill (1,0) circle (2pt);
\draw (0,1)--node[above] {{\scriptsize$1$}}(1,1);
\draw[line width=1.5pt,gray]
(0,0)--node[left] {{\scriptsize$2$}}(1,1);
\draw (0,0)--node[below] {{\scriptsize$1$}}(1,0);
\end{tikzpicture}
&
$2\times2\times1$
&
$\langle z_0z_1,z_0z_3,z_1z_2,z_1z_3,z_2z_3\rangle$
\\
\hline
\hline

&
&
$56$ ($\mathbb A^1$-family: $8$)
&
\\
\hline
\end{tabular}
\caption{Classification of the $T$-fixed rational quartic curves on $Q$.}
\label{tab:traq}
\end{table}


\subsection{Application of the localization formula}
Now we compute the localization contribution of $\bM_4^T$.

In this subsection, representation characters are recorded after composition with the torus inversion $( (t_1,t_2)\mapsto(t_1^{-1},t_2^{-1}) )$, relative to the action fixed in subsection \ref{subdefq}.
\subsubsection{Contribution of isolated $T$-fixed points in $\bR_4$}
\begin{proposition}\label{prop:equ}
Let $C\subset Q\subset \PP^4$ be a $T$-fixed non-degenerate rational curve. Let $U_5=\rH^0(\PP^4, \cO_{\PP^4}(1))$, and let $V_5=\rH^0(Q, I_{C/Q}(2))$ be the space of quadratic generators of $I_{C/Q}$. Furthermore, let $R_8$ denote the eight-dimensional space of linear syzygies described in Corollary \ref{syre}. Then 
\begin{enumerate}
\item the restricted insertions along $F=\mathcal O_C$ are
\[
\tau_0(h_2)|_F
=
-\frac{1}{2}c_1(V_5),
\;
\tau_1(h_1)|_F
=
\frac{1}{2}c_1(V_5).
\]
\item the self-Euler pairing of $F=\cO_C$ is 
\[
\chi(F,F)
=
1+V_5^\vee(V_5-\mathrm{Sym}^2 U_5+1)
+\left(\frac{R_8}{V_4}\right)^\vee
\left[
\frac{R_8}{V_4}
+V_4(\mathrm{Sym}^2 U_5-U_5-V_5)
\right].
\]
\end{enumerate}
\end{proposition}
\begin{proof}
Let $W_2$ denote the two-dimensional $T$-representation in the equivariant resolution of $I_{C/Q}$ in \eqref{eqres1}:
\begin{equation}\label{eqres}
0
\lr
W_2\otimes\cE(-2)
\lr
V_5\otimes\cO_Q(-2)
\lr
I_{C/Q}
\lr0.
\end{equation}
By the proof of Corollary \ref{syre} and the structure sequence
$0\lr
\cE
\lr
V_4\otimes\cO_Q
\lr
\cE(1)
\lr0$ of the spinor bundle $\cE$, the space of linear
syzygies $R_8$ is precisely
\[
R_8=W_2\otimes H^0(\cE(1))=W_2\otimes V_4.
\]
Hence in $K_T(Q)$, we have
\begin{equation}\label{eq:ocn}
[F]
=
1-[V_5][\cO_Q(-2)]
+\left[\frac{R_8}{V_4}\right][\cE][\cO_Q(-2)].
\end{equation}
For Part (1), retaining only the terms of total degree at most \(3\) needed for the insertion computations, we obtain
\[
\ch_{\leq 3}(F)
=
\bigl(2h_1^2-c_1(W_2)h_1\bigr)
+
\left(
-\frac52h_1^3+2c_1(W_2)h_1^2
\right).
\]
Since $K_Q=\cO_Q(-3)$, $\deg(Q)=2$, and $c_1(V_5)=2c_1(W_2)$
by \eqref{eqres}, we have
\[
\tau_0(h_2)|_F
=
\int_Q
h_2\,
\left\{
\ch(F)\td(\mathcal O_Q(-3))^{-1}
\right\}_2
=
-\int_Q c_1(W_2)h_1h_2
=
-\frac12 c_1(V_5),
\]
\[
\tau_1(h_1)|_F
=
\int_Q
h_1\,
\left\{
\ch(F)\td(\mathcal O_Q(-3))^{-1}
\right\}_3
=
\int_Q
\frac12 c_1(W_2)h_1^3
=
\frac12 c_1(V_5).
\]
Part (2) follows from \eqref{eq:ocn}, the equivariant Euler
characteristics of the bundles $\cO_Q(k)$, $\cE(k)$, and the bilinearity of the Euler pairing.
\end{proof}
\begin{remark}
The space \(W_2=\frac{R_8}{V_4}\) in Proposition \ref{prop:equ} can be determined directly from the multiplication map \eqref{eq:mul} by using \texttt{Macaulay2} \cite{M2}. The lists are given in Table~\ref{tab:W2-weights}.
\end{remark}
\begin{table}[ht]
\centering
\renewcommand{\arraystretch}{1.35}
\begin{tabular}{|c|c|c|c|c|c|c|c|}
\hline
No.& 1 & 2 & 3 & 6 & 8 & 9 & 10 \\
\hline
\hline
$W_2$
&
$t_1^{-2}t_2+t_1^{-2}t_2^{-1}$
&
$t_1^{-2}t_2^{-1}+t_2^{-1}$
&
$t_1^{-1}t_2^{-2}+t_1^{-1}$
&
$t_1^{-1}+t_2$
&
$2t_1^{-1}$
&
$t_1+t_1^{-1}$
&
$t_1+t_2$
\\
\hline
\end{tabular}
\caption{The $T$-characters $W_2$ associated with the representative ideals.}
\label{tab:W2-weights}
\end{table}

\begin{proposition}\label{pr:deg}
Let $C\subset S=Q\cap H$ be a degenerate $T$-fixed quartic rational curve in
$H=\PP(W^\vee)$. Let
\[
V_1 = \rH^1\!\left(S,I_{C/S}(1)\right),
\;
V_3 = \rH^2\!\left(S,I_{C/S}(-1)\right)
\]
be the corresponding $T$-representations. Then for
$L=[\cO_H(1)]$, we have
\begin{enumerate}
\item $[\cO_C]
=
1+V_1L^{-1}-(1+V_1W)L^{-2}+V_3L^{-3}$ in $K_T(H)$.
\item
$\left.\tau_0(h_2)\right|_F=-c_1(V_1)$ and
$\left.\tau_1(h_1)\right|_F=2c_1(V_1)$.
\item
$
\chi(\cO_C,\cO_C)
=
\chi^H\!\left(
\frac{[\cO_C]^\vee[\cO_C]}{1+L}
\right)$.

\end{enumerate}
\end{proposition}

\begin{proof}
By Lemma 4.1 of \cite{CCM16}, we have the equivariant resolution of $I_{C/S}$:
\begin{equation}\label{eq:bel}
0\longrightarrow
V_3\otimes\mathcal O_H(-3)
\longrightarrow
V_1\otimes\Omega_H(-1)
\longrightarrow
I_{C/S}
\longrightarrow0,
\end{equation}
where the descriptions of $V_1$ and $V_3$, $\dim V_i=i$ follow immediately by
twisting the resolution and taking cohomology. By the Euler sequence of the projective space $H$,
\[
[\Omega_H(-1)]=WL^{-2}-L^{-1}
\]
in $K_T(H)$. Hence
\[
[I_{C/S}]
=
V_1(WL^{-2}-L^{-1})-V_3L^{-3}.
\]
Since $[\mathcal O_S]=1-L^{-2}\in K_T(H)$, this proves (1).
The proof of (2) is identical to the non-degenerate case
(Proposition \ref{prop:equ}) and is omitted. For (3), let $i\colon S\hookrightarrow H$, $j\colon S\hookrightarrow Q$ and put $G=\cO_C\in K_T(S)$ and $A=i_*G\in K_T(H)$.
Since $N_{S/H}\cong\mathcal O_S(2)$ and
$N_{S/Q}\cong\mathcal O_S(1)$, the equivariant self-intersection
formula gives
\[
A^\vee A
=
i_*\!\left(G^\vee G(1-L^2)\right),
\;
(j_*G)^\vee(j_*G)
=
j_*\!\left(G^\vee G(1-L)\right).
\]
Since $1-L^2=(1-L)(1+L)$, the claim follows from the projection
formula.\end{proof}
\begin{example}
The characters of $V_1$ and $V_3$ in Proposition \ref{pr:deg} can be read off directly from the
weights of the equivariant short exact sequence \eqref{eq:bel}.  For example, consider the ideal
\[
I_C=\langle z_2,z_4^3,z_3z_4^2,z_3^2z_4\rangle
\]
of No. 4 in the TABLE \ref{tab:traq}.
Here $H=\{z_2=0\}$ and $W=\langle z_0,z_1,z_3,z_4\rangle$.  After twisting \eqref{eq:bel} by $\cO_H(3)$ and taking global sections, we have
\[
0\longrightarrow V_3
\longrightarrow V_1\otimes\Lambda^2W
\longrightarrow\rH^0(I_{C/S}(3))=
\langle z_4^3,z_3z_4^2,z_3^2z_4\rangle
\longrightarrow0.
\]
As we compare the weights of $\Lambda^2W$ with those of $\rH^0(I_{C/S}(3))$, the unique character of $V_1$ is
\[
V_1=t_1^{-3}t_2^{-1}.
\]
Hence the weights of the kernel $V_3$ are $V_3=t_1^{-1}t_2^{-1}+t_1^{-3}t_2^{-1}+t_1^{-5}t_2^{-1}$.
For No. 7 in the TABLE \ref{tab:traq}, we have $V_1=t_1^{-1}t_2$ and $V_3=t_1^{-1}t_2^3+t_1^{-1}t_2+t_1^{-1}t_2^{-1}$.
\end{example}

\subsubsection{Contribution of non-isolated component in $\bR_4^T$}\label{sssecp1}
The one-parameter family in TABLE \ref{tab:traq} extends into a non-CM curve $C_\infty$ (Remark \ref{re:fla}) and therefore we require an elementary modification to obtain a family of stable sheaves. As carrying out this modification equivariantly is rather cumbersome, we are willing to reinterpret the family as a tautological extension family (Proposition \ref{lm:mop1}).
\begin{lemma}\label{lm:idf}
Let $L$ be the $T$-fixed line defined by $I_L=\langle z_2, z_3, z_4\rangle$ and let $C_3$ be the $T$-fixed curve defined by  $I_{C_3}=\langle z_4,z_2^2,z_2z_3\rangle$. Then, there is a two-dimensional $T$-eigenspace
\[
V_\eta\subset \Hom_Q(I_{C_3/Q},\mathcal O_L(-1))
\]
such that the surjective homomorphisms in
$\PP_{[\alpha:\beta]}(V_\eta)\setminus\{[0:1]\}$ parametrize the one-parameter family for $t=\frac{\beta}{\alpha}$ in Proposition \ref{onef}.
\end{lemma}

\begin{proof}
Let $u$ be a local generator of $\mathcal O_L(-1)$. By the similar argument in the proof of Proposition \ref{mult4line}, we know that among homomorphisms $\phi: I_{C_3/Q} \lr \mathcal O_L(-1)$, the unique two-dimensional $T$-eigenspace is
\[
V_{\eta}=\langle \phi_0,\phi_1\rangle,
\]
where $\phi_0(z_4)=u$ and $\phi_1(z_2^2)=z_0u$
and all other images are zero. Note that both $\phi_0$ and $\phi_1$ have $T$-character $\eta=t_1t_2$. Let
\[
\phi_{\alpha,\beta}
=\alpha\phi_0+\beta\phi_1,
\;[\alpha:\beta]\in\PP(V_\eta).
\]
For $\alpha\neq0$, the map $\phi_{\alpha,\beta}$ is surjective and
\[
\ker(\phi_{\alpha,\beta})
=
\langle z_2z_3,z_2z_4,z_3z_4,z_4^2,\alpha z_2^2-\beta z_0z_4\rangle.
\]
By letting $t=\beta/\alpha$, we have $\ker(\phi_{\alpha,\beta})=
I_{C_t/Q}$ in Proposition \ref{onef}. Thus the surjective homomorphisms are parametrized by
$\mathbb P(V_\eta)\setminus\{[0:1]\}=\mathbb A^1_t$.
\end{proof}
The $\AA^1$-family of CM-curves in Lemma~\ref{lm:idf} can be extended to a $\PP^1$-family of stable sheaves as follows.
\begin{proposition}\label{lm:mop1}
In the situation of Lemma \ref{lm:idf},
\begin{enumerate}
\item there exists an isomorphism 
$$\Ext_Q^1(\mathcal O_{C_3},\mathcal O_L(-1))\cong\Hom_Q(I_{C_3/Q},\mathcal O_L(-1)).$$
Hence the tautological family of stable sheaves $\cF_{\bZ}$ parameterized by $\bZ=\PP(V_\eta)$ fits into the short exact sequence:
\begin{equation}\label{unz}
\ses{\mathcal O_L(-1)\boxtimes(\cO_{\bZ}(1)\otimes\CC_{\eta^{-1}})}{\cF_\bZ}{\mathcal O_{C_3}}.
\end{equation}
\item Let $z=c_1(\mathcal O_{\bZ}(1))$ and $u_i=c_1(t_i)$, $i=1,2$. Then 
$$
\left.\tau_0(h_2)\right|_{\bZ}=3u_1+u_2\; \mathrm{and}\;
\left.\tau_1(h_1)\right|_{\bZ}=z-3u_1-u_2.$$
\item The self Euler pairing of $\cF_\bZ$ is 
\[
\begin{aligned}
\chi_\bZ(\mathcal F_\bZ,\mathcal F_\bZ)
={}&
\Bigl(
2+t_1^4+t_1^4t_2^2+t_1^4t_2^{-2}
+t_1^3t_2^3
+2t_1^3t_2
+2t_1^3t_2^{-1}
\\
&\qquad
-3t_1^2-t_1^2t_2^{-2}
+t_1t_2^3
-4t_1t_2
-2t_1t_2^{-1}
-t_2^2
+t_1^{-1}t_2^3
-t_1^{-1}t_2
\Bigr)
\\
&+
[\mathcal O_\bZ(-1)]
\Bigl(
t_1t_2+t_1^5t_2+t_1^5t_2^3+t_1^5t_2^{-1}
-2t_1^3t_2-t_1^3t_2^{-1}-t_1^2t_2^2
\Bigr)
\\
&+
[\mathcal O_\bZ(1)]
\Bigl(
t_1^2+t_1^2t_2^{-2}+t_1t_2
-2-t_2^{-2}
\Bigr).
\end{aligned}
\]
\end{enumerate}
\end{proposition}
\begin{proof}
(1) The isomorphism result comes from $\bR\Gamma(\mathcal O_L(-1))=0$ and the short exact sequence
$0\longrightarrow I_{C_3/Q}
\longrightarrow \mathcal O_Q
\longrightarrow \mathcal O_{C_3}
\longrightarrow0$. Since $V_\eta$ has character
$\eta$, the equivariant tautological subbundle on
$\bZ=\mathbb P(V_\eta)$ is $\cO_\bZ(-1)\otimes\CC_{\eta}$.
Hence we obtain the short exact sequence \eqref{unz} (cf. \cite[Example 2.1.12]{HL}). Note that each non-split extension sheaf over $\bZ$ is stable by Lemma 4.7 in \cite{CCM16}.

(2) From \eqref{unz}, we have $\ch^T(\mathcal F_{\bZ})=e^{z-u_1-u_2}\ch^T(\mathcal O_L(-1))+\ch^T(\mathcal O_{C_3})$. Using the $T$-equivariant resolutions of $\mathcal O_L(-1)$ and $\mathcal O_{C_3}$, we obtain the
Chern character of $\cF_{\bZ}$ up to degree $3$:
\[
\begin{aligned}
\ch_{\leq 3}(\mathcal F_{\bZ})
={}&
\left(
2h_1^2
+(3u_1+u_2)h_1
+u_1^2-u_2^2
\right)
\\
&+
\left(
-\frac52 h_1^3
+\left(
\frac z2-6u_1-2u_2
\right)h_1^2
+\cdots
\right).
\end{aligned}
\]
Substituting these ones into the insertion formulas and integrating over $Q$ yields the result.

(3) Since $[\cF_\bZ]=\eta^{-1} [\cO_\bZ(1)][\mathcal O_L(-1)]+ [\mathcal O_{C_3}]$, we have
\begin{equation}\label{bipa}
\begin{aligned}
\chi_\bZ(\cF_\bZ,\cF_\bZ)
&={}
\chi(\mathcal O_L(-1),\mathcal O_L(-1))
+\chi(\mathcal O_{C_3},\mathcal O_{C_3})\\
&+\eta\,[\cO_\bZ(-1)]\chi(\mathcal O_L(-1),\mathcal O_{C_3})+\eta^{-1}[\cO_\bZ(1)]\,
\chi(\mathcal O_{C_3},\mathcal O_L(-1))
\end{aligned}
\end{equation}
by the bilinearity of the equivariant Euler pairing. Here, the equivariant Euler characteristics appearing in \eqref{bipa} are computed using Thomason's localization formula (\cite{Tho92}).
\end{proof}
\begin{remark}\label{re:nor}
Since $\chi(\mathcal O_L(-1))=0$ and $\chi(\mathcal O_{C_3})=1$, we have
\[
\mathrm{det} R\pi_{\bZ*}(\cF_\bZ)=
(\cO_\bZ(1)\otimes\CC_{\eta^{-1}})^{\chi(\mathcal O_L(-1))}
\otimes
\cO_\bZ^{\chi(\mathcal O_{C_3})}
\cong
\cO_\bZ.
\]
Hence the restricted universal sheaf $\cF_{\bZ}$ in \eqref{unz} is a normalized one.
\end{remark}
\subsubsection{Contribution of the fixed component $\bE_4^T$}
As we have seen the global description of $\bE_4$ in (3) of Proposition \ref{4mpro}, the torus action on the elliptic component $\bE_4$ is explicit, and its fixed sheaves are described by the pair $(p, E)$ of an elliptic quartic curve $E$ and a point $p$ on $E$. Consequently, their localization contributions are obtained by a relatively straightforward computation.
\begin{proposition}\label{fixe}
Let $F\in \bE_4^T$ be a $T$-fixed sheaf which fits into the $T$-equivariant exact sequence ((2) of Proposition \ref{4mpro})
\begin{equation}\label{eq412}
0\longrightarrow \mathcal O_E
\longrightarrow F
\longrightarrow \mathbb C_{p_u}\otimes \CC_\eta
\longrightarrow0,
\end{equation}
where $E$ is defined by $\mathbf{h}$, $\mathbf{q}$.
Then for $h=\wt(\mathbf{h}), q=\wt(\mathbf{q})$, and $u=\mathrm{wt}(\cO_Q(1)|_{p_u})$,
\begin{enumerate}
\item $\eta=hq$.
\item The restricted insertions are
\[\tau_0(h_2)|_F=-2c_1(h)-c_1(q),\; \tau_1(h_1)|_F=c_1(u)+2c_1(h)+2c_1(q).\]
\item The self-Euler pairing of the fixed sheaf $F$ is 
\begin{equation}\label{eq:eext}
\begin{aligned}
\chi_T(F,F)
={}&
3+\frac{h}{q}U_5
-\frac{2}{q}\bigl(\operatorname{Sym}^2 U_5-1\bigr)
-\frac{2}{h}U_5
+\frac{1}{hq}\bigl(\operatorname{Sym}^3 U_5-U_5\bigr)
\\
&\quad
-qu-hu^2+u^3
-\frac{u^3}{hq}
+\frac{u^2}{q}
+\frac{u}{h}
+\frac{1}{1-u^2}
\prod_{ v\neq u}
\left(1-\frac{u}{v}\right). 
\end{aligned}
\end{equation}
\end{enumerate}
\footnote{Maybe $V_5$ here should be replaced by $U_5$ ?}

\end{proposition}
\begin{proof}
For (1), from the $T$-equivariant Koszul resolution of the curve $E$ (Remark \ref{remcm}), we have $\chi_T(\mathcal O_E)=1-hq$. Also, since
$\bR\Gamma(F)\cong\mathbb C[0]$ (\cite[Lemma 4.1]{CCM16}) and the canonical section $\cO_E\hookrightarrow F$ is $T$-invariant, we have $\chi_T(F)=1$.
Taking equivariant Euler characteristics in \eqref{eq412} gives $\eta=hq$.

For (2) and (3), one can immediately obtain the results from \eqref{eq412}, Remark \ref{remcm}, and bilinearity of the Euler pairing.
\end{proof}

\begin{remark}
By the construction of the short exact sequence \eqref{eq412} in Proposition \ref{fixe}, the
$T$-linearization of $F$ is a normalized one.
\end{remark}

\subsection{Cao-Toda's conjecture and genus $1$ BPS-number}
The result can be verified directly by \texttt{Macaulay2} \cite{M2} computation of the contribution of every fixed locus. However, the geometric symmetries described above allow us to carry out the computation much more efficiently.

As discussed in Section \ref{sub:rfixed}, the rational component contains $48$ isolated fixed sheaves together with eight fixed components isomorphic to $\mathbb P^1$. By the compatibility of the localization data with the $D_4$-action (Lemma \ref{tacq}), it is enough to compute the contribution of a single representative from each $D_4$-orbit. Consequently, for the isolated fixed loci on the rational component, only seven representative cases need to be computed explicitly in
\texttt{Macaulay2} \cite{M2}. Similarly, the eight $\mathbb P^1$-components form a single $D_4$-orbit, so it suffices to compute one representative component.

On the other hand, the elliptic component contains $128$ isolated fixed
sheaves, eight of which occur as boundary points of the
$\mathbb P^1$-components considered above. Thus $120$ isolated fixed
sheaves remain. It is enough to compute the contributions of the $30$
fixed sheaves whose support curve passes through a fixed point; the contributions of the remaining fixed
sheaves are obtained from these by the transform of the cyclic subgroup $\langle r\rangle$ of $D_4$ (Lemma \ref{tacq}). After all, we have the following results.
\begin{proposition}\label{pdin}
The primary/descendent invariants of $Q$ with degree $d=4$ are
\[
\langle \tau_0(h_{2})\rangle_{4}=-265940,\;\langle \tau_1(h_{1})\rangle_{4}=-73360.
\]
\end{proposition}
Note that for degree $d\leq 3$, one can directly compute the invariants by using the global description of moduli spaces of stable sheaves on $Q$ (\cite{CHY23}). Combining these computations with the results in Proposition \ref{pdin}, we have TABLE \ref{tab:invariants_d1to4}.
\begin{table}[htbp] 
\begin{center} \begin{tabular}{|l||M{2.2cm}|M{2.2cm}|} \hline $d$ & $i=0$ & $i=1$\\ \hline\hline $1$ & $10$ & $10$\\ \hline $2$ & $-175$ & $0$\\ \hline $3$ & $5870$ & $-1310$\\ \hline $4$ & $-265940$ & $-73360$\\ \hline \end{tabular} \caption{The invariants $\langle \tau_i(h_{2-i})\rangle_{d}$ for $d\leq 4$} \label{tab:invariants_d1to4} 
\end{center} \end{table}

\begin{remark}
From Proposition \ref{pdin}, the Cao-Toda conjecture (\cite[Conjecture 0.2]{CMT18} and \cite[Conjecture 0.2]{CT21}) holds in degrees up to $4$ under the assumption that the genus one BPS numbers satisfy $n_{1,d}=0$, $d\leq 3$ and $n_{1,4}=45150$.
\end{remark}

\bibliographystyle{alpha}
\bibliography{Library.bib}

\end{document}